\documentclass[12pt,a4paper]{article}

\usepackage[dvips]{graphicx}
\usepackage{amsmath}
\usepackage{amssymb,latexsym}
\usepackage{mathrsfs}
\usepackage{psfrag}
\usepackage[all]{xy}
\usepackage{color}
\usepackage{verbatim}
\usepackage{graphicx,geometry,color,epsfig,epstopdf}
\usepackage{caption}
\usepackage{subcaption}
\usepackage{setspace}
\usepackage{soul}

\newtheorem{prop}{Proposition}[section]
\newtheorem{theo}[prop]{Theorem}
\newtheorem{cor}[prop]{Corollary}
\newtheorem{ex}[prop]{Example}
\newtheorem{exs}[prop]{Examples}
\newtheorem{defn}[prop]{Definition}

\newtheorem{rems}[prop]{Remarks}
\newtheorem{lem}[prop]{Lemma}

\newcommand{\C}{\mbox{\boldmath $\mathbb{C}$}}
\newcommand{\R}{\mbox{\boldmath $\mathbb{R}$}}
\newcommand{\K}{\mbox{\boldmath $\mathbb{K}$}}
\newcommand{\N}{\mbox{\boldmath $\mathbb{N}$}}

\newcommand{\ord}{\textrm{ord}}

\newenvironment{proof}
{\begin{trivlist} \item[\hskip \labelsep {\bf Proof}\hspace*{3 mm}]}
	{\hfill$\Box$\end{trivlist}}
\newenvironment{acknow}
{\begin{trivlist} \item[\hskip \labelsep {\bf Acknowledgments.}]}
	{\end{trivlist}}

\begin{document}

\title{On geometric invariants of singular plane curves}
\author{J. W. Bruce, M. A. C. Fernandes and F. Tari}

\maketitle

\begin{abstract}
Given a germ of a smooth plane curve $(\{f(x,y)=0\},0)\subset (\K^2,0), \K=\R, \C$, with an isolated singularity, we define two invariants $I_f$ and $V_f\in \N\cup\{\infty\}$  which count the number of inflections and vertices (suitably interpreted in the complex case) concentrated at the singular point; the first is an affine invariant and the second is invariant under similarities of $\R^2$, and their analogue for $\C^2$. We show that for almost all representations of $f$, in the sense that their complement is of infinite codimension, these invariants are finite. Indeed when the curve has no smooth components they are always finite and bounded and we can be much more explicit about the values they can attain; the set of possible values is of course an analytic invariant of $f$. We illustrate our results by computing these invariants for Arnold's ${\cal K}$-simple singularities as well as singularities that have ${\cal A}$-simple parametrisations. We also obtain a relationship between these invariants, the Milnor number of $f$ and the contact of the curve germ with its \lq osculating circle\rq.
\end{abstract}

\renewcommand{\thefootnote}{\fnsymbol{footnote}}

\section{Introduction}\label{sec:intro}
Suppose given a germ of a plane curve with an isolated singularity. The plane can be endowed with an affine, Euclidean (real plane) or Hermitian (complex plane) structure, and it is natural to ask what geometry emerges from the singularity when the curve is deformed. In this paper we count the number of inflections and vertices \lq concentrated\rq\ at the singular point. 

Inflections and vertices of plane curves capture key aspects of their geometry. 
Inflections are classical objects of study in plane projective geometry: the inflections of a nonsingular algebraic plane curve $F=0$  are the points of intersection of the curve with its  Hessian $H(F)=0$ so that generically a projective curve of degree $d$ has $3d(d-2)$ inflections. Plucker's equations then indicate amongst other things how this number may change as the curve acquires singular points. 
Viro \cite{viro} extended Plucker's formula to include  the number sextatic points;  these are smooth points where the osculating conic intersects the curve with intersection multiplicity $6$ or higher.
An upper bound  for the number of vertices of a real algebraic curve  is also given in \cite{viro} using the complexification of the curve. Vertices here can be defined in terms of the contact of the curve with complex circles: these are the conics which, when projectivised, pass through the \lq circular points at infinity\rq \, $[1:\pm i:0]$ (\cite{viro}).

While primarily motivated by the real case it is most efficient to start with germs of singular holomorphic curves $f(x,y)=0$ in $\C^2$; we denote this germ by $(C_f,0)$. We define the curvature of a regular curve in $\C^2$ endowed with the Hermitian inner product. A vertex of a plane curve turns out to be a point where the curvature has an extremum, and inflections are points where the curvature vanishes, all as in the real case. We define numbers $I_f$ and $V_f$ in $\N\cup\{\infty\}$ which count the number of inflections and vertices concentrated at the singular point. 
We establish relations between the above numbers and those associated to the irreducible branches of $f$. Given a parametrisation $\gamma$ of an irreducible germ $f$, we 
establish relations between the number of inflections and vertices of $f$, those of $\gamma$ defined in \cite{DiasTari}, and the Milnor number $\mu(f)$ of $f$. 
 
We show that any germ $f$ with an isolated singularity is $\mathcal K$-equivalent to a germ  $\tilde f$  
with $I_{\tilde f}$ and $V_{\tilde f}$ finite in the real or complex case. In fact the set of representations of $f$ for which this is not true is of infinite codimension in the sense of Tougeron, see \cite{wall}, p. 513.
We prove that $f$ admits deformations with only ordinary inflections and vertices. In the complex case, when $C_f$ has no smooth components, we show that $I_f$ and $V_f$ are bounded and provide detailed information for the range of values they can attain. As an illustration  we determine the values $I_f$ and $V_f$ can take when $f$ is one of Arnold's simple singularities as well as the ${\cal A}$-simple singularities $(\C,0)\to (\C^2,0)$ in \cite{BruceGaffney}.

One of our results (Theorem \ref{teo:princ} part (1)) concerning inflections is contained in the work of Wall in \cite{wallflat}. Just as the computations of $I_f$ lead to Plucker type formulae and information on the dual as in \cite{WallDuality}, so the calculations on $V_f$ give results on the singularities of the evolute.

This paper is organized as follows. In \S \ref{sec:prel} we give some preliminaries and results about vertices of regular complex analytic curves. 
In \S \ref{sec:GeometricInvariants}, we define $I_f$ and $V_f$ and give some of their properties.  
We calculate in \S\ref{sec:exam} the range of the values of $I_f$ and $V_f$ for curves with simple singularities.
We return to the real case in \S \ref{sec:vertR}, deducing our results from the complex case using finite determinacy results. In what follows all our defining functions $f:(\C^2,0)\to (\C,0)$ are ${\cal K}$-finite, that is $f$ has no repeated factors.


\section{Preliminaries} \label{sec:prel}
 
\subsection{Vertices and inflections of regular curves}
Recall that two regular curves in $\C^2$, one parametrised by $\gamma: U\to \mathbb \C^2$ and  the other given by the equation $g=0$, with $g:V\to \C$, $U\subset \C$ and $V\subset \C^2$ open, have {\it $k$-point contact} at $t_0$, $k\ge 1$, if $\gamma(t_0)\in V$, $(g\circ \gamma)(t_0)=0$, $(g\circ\gamma )^{(i)}(t_0)=0$ for $i=1,..,k-1$ and $(g\circ\gamma )^{(k)}(t_0)\ne 0$. They have {\it at least $k$-point contact} if we drop the last condition.

Inflections of regular holomorphic curves in $\C^2$ (or algebraic curves in $\C\mathbb P^2$) are points where the tangent line has at least 3-point contact with the curve. An extension of the traditional concept of a vertex to the complex case can also be defined using the contact of the curve with circles, see \cite{viro}. We now give some details about such contact and define the curvature of a regular holomorphic curve.

We suppose that the complex plane $\C^2$ is endowed with the Hermitian inner product $$\langle {\bf x}_1,{\bf x}_2\rangle=x_1\overline{x_2}+y_1\overline{y_2}$$ for 
${\bf x}_1=(x_1,y_1)$ and ${\bf x}_2=(x_2,y_2) $ in $\C^2$.

Let $\gamma(t)=(x(t),y(t))$, $t\in U\subset \C$, be a parametrisation of a regular holomorphic curve in $\C^2$, where $U$ 
is an open set in $\C$. 
The tangent vector to $\gamma$ at $\gamma(t)$ is $\gamma'(t)=(x'(t),y'(t))$ and a normal vector (with respect to the Hermitian inner product) is one parallel to $\gamma'(t)^{\perp}=(- \overline{y'(t)},\overline{x'(t)})$.

The tangent line of $\gamma$ at $t_0$ is $-y'(t_0)(x-x(t_0))+x'(t_0)(y-y(t_0))=0$, so $t_0$ is an inflection point of $\gamma$ if and only if
$$
(x'y''-y'x'')(t_0)=0.
$$

A complex circle in $\C^2$ is a conic with equation 
$$
(x-a)^2+(y-b)^2=c^2,
$$
where $a,b,c\in \C$ (see \cite{Mitchell} for more on complex circles); projectivised these are the conics passing through the circular points at infinity $(1:\pm i:0)$. We shall refer to ${p}=(a,b)$ as the centre of the circle and $c$ as its radius. The equation of a circle can also be written as 
$
\langle {\bf x}-p,\overline{{\bf x}-p}\rangle=c^2, 
$
 where $\overline{\bf x}=(\overline{x},\overline{y})$.

The contact of the curve $\gamma$ at $t_0$ with circles of centre ${p}=(a,b)$ passing through $\gamma(t_0)$ is measured by the singularities of the (contact) function 
$$
	d_{p}(t)=(x(t)-a)^2+(y(t)-b)^2,
$$
that is, by the order of  the function $d_{p}$ at $t_0$. We have
$$
\begin{array}{rcl}
	\frac12 d'_{p}(t_0)&=&x'(t_0)(x(t_0)-a)+y'(t_0)(y(t_0)-b),\\
	\frac12 d''_{p}(t_0)&=&x''(t_0)(x(t_0)-a)+y''(t_0)(y(t_0)-b)+ x'(t_0)^2+ y'(t_0)^2,\\
	\frac12 d'''_{p}(t_0)&=&x'''(t_0)(x(t_0)-a)+y'''(t_0)(y(t_0)-b)+3x'(t_0)x''(t_0)+ 3y'(t_0)y''(t_0).
\end{array}
$$

It follows that $d'_{p}(t_0)=0$ if and only if 
$x(t_0)-a=-\lambda y'(t_0)$ and $y(t_0)-b=\lambda x'(t_0)$ for some $\lambda \in \C$.
This means that the circle is tangent to the curve at $t_0$ if and only if ${p}=(a,b)$ 
belongs to the line parallel to $\overline{\gamma'(t)^{\perp}}$ and passing through $\gamma(t_0)$.
We call such line the {\it conjugate normal line} to $\gamma$.

Suppose that $t_0$ is not an inflection point of $\gamma$. Then $d'_{p}(t_0)=d''_{p}(t_0)=0$ if and only if  
$x(t_0)-a=-\lambda y'(t_0)$ and $y(t_0)-b=\lambda x'(t_0)$ with 
$
\lambda=\frac{x'(t_0)^2+y'(t_0)^2}{(x'y''-y'x'')(t_0)}
$, equivalently,
\begin{equation}\label{eq:evolute}
{p}=	(a,b)=(x(t_0),y(t_0))+\frac{x'(t_0)^2+y'(t_0)^2}{(x'y''-y'x'')(t_0)}(-y'(t_0),x'(t_0)).
\end{equation}
We call the circle of centre ${p}$ in (\ref{eq:evolute}) the {\it osculating complex circle} of $\gamma$ at $t_0$.

Suppose that $t_0$ is not an inflection point and $d'_{p}(t_0)=d''_{p}(t_0)=0$. Then $d'''_{p}(t_0)=0$ 
if and only if 
\begin{equation}\label{eq:vertex}
	(x'''y'-y'''x')(x'^2+y'^2)+
	3(x'x''+y'y'')(x'y''-x''y')=0
\end{equation}
at $t_0$.

\begin{defn}\label{def:vert}
	Let $\gamma(t)=(x(t),y(t))$ be a regular holomorphic curve in $\C^2$. 
	We say that $\gamma(t_0)$ is a vertex of $\gamma$ 
	if equation $(\ref{eq:vertex})$ is satisfied at $t_0$, that is, there exists a circle that has at least $4$-point contact with $\gamma$ at $\gamma(t_0)$. {\rm (}This agrees with the definition in {\rm \cite{viro}.)}
	
An inflection {\rm(}resp. vertex{\rm)} is called a simple inflection {\rm(}resp. vertex{\rm)} if the contact of the tangent line {\rm(}resp. osculating complex circle{\rm)} with the curve is precisely $3$-point {\rm(}resp. $4$-point{\rm)}. 
\end{defn}

We can rewrite (\ref{eq:evolute}) as follows. Write
$$\begin{array}{rcl}
	n(t_0)&=&\frac1{\langle \gamma'(t_0)^{\perp}, \overline{\gamma'(t_0)^{\perp}}\rangle^{\frac12} }\overline{\gamma'(t_0)^{\perp}}\vspace{0.2cm}\\
	&=&\frac1{\langle \gamma'(t_0), \overline{\gamma'(t_0)}\rangle^{\frac12} }\overline{\gamma'(t_0)^{\perp}}\vspace{0.2cm}\\
	&=&
	\frac{1}{(x'^2(t_0)+y'^2(t_0))^{\frac12}}(-y'(t_0),x'(t_0)).
\end{array}
$$
Clearly $\langle n(t_0),\overline{n(t_0)}\rangle=1$, that is, $n(t_0)$ is of unit length, and 
$d'_{p}(t_0)=d''_{p}(t_0)=0$ if and only if
$$
{p}=(a,b)=\gamma(t_0)+\frac{1}{\kappa(t_0)}n(t_0)
$$
where
\begin{equation}\label{eq:curvature}
	\kappa=\frac{x'y''-y'x''}{(x'^2+y'^2)^{\frac32}}.
\end{equation}

\begin{defn}
	Let $\gamma(t)=(x(t),y(t))$ be a regular holomorphic curve in $\C^2$. We call 
	$\kappa(t)$ in $(\ref{eq:curvature})$ the curvature of $\gamma$ at $t$ 
	and the curve parametrised by 
	$$
	e(t)=\gamma(t)+\frac1{\kappa(t)}n(t)
	$$ 
	the evolute of $\gamma$.
\end{defn}

\begin{prop} \label{theo:kappaRegularC}
	Let $\gamma(t)=(x(t),y(t))$ be a regular holomorphic curve in $\C^2$.
	
	{\rm (1)} The curve $\gamma$ has an inflection at $t_0$ if and only if 
	$\kappa(t_0)=0$. It has a vertex at $t_0$ if and only if $\kappa'(t_0)=0$.
	
	{\rm (2)} $\gamma$ is a part of a line if and only if $\kappa\equiv 0$.
	
	{\rm (3)} $\gamma$ is a part of a complex circle if and only if $\kappa\equiv const\ne 0$.
	
	{\rm (4)} The evolute of $\gamma$ is the envelope of its conjugate normal lines. It is singular at $t_0$ if and only if $t_0$ is
	a vertex of $\gamma$.
	
{\rm (5)} The dual of $\gamma$ has an ordinary cusp $(A_2)$ at a point corresponding to a simple  inflection.
	
{\rm (6)} The evolute of $\gamma$ has an ordinary cusp $(A_2)$ at a point corresponding to a simple vertex which is not an inflection.
\end{prop}
\begin{proof}
The results are elementary and proved in the usual way.
\end{proof}

\begin{rems}
	{\rm
		1. Viro \cite{viro} considered unit speed parametrizations, that is those with $\langle {\bf x}',\overline{{\bf x}'}\rangle=1$. Any regular curve can be re-parametrised in this way and of course then $\kappa=x'y''-x''y'$. 
		
		2. Ness \cite{Ness1,Ness2}  endowed a holomorphic curve in $\C^2$ with the metric induced by the Fubini-Study
		and studied its Gaussian curvature away from singular points.		
		
	3. When $\K=\R$ the group preserving the circles is the similarity group. In the complex case, if we projectivise, it is the group preserving the circular points at infinity. We can think of the subgroup fixing those points as the analogue of the orientation preserving transformations, and the coset exchanging them corresponding to orientation reversing transformations. One can check that the group fixing the circular points at infinity and the origin consists of the linear maps $(x,y)\mapsto (ax+by,-bx+ay)$, with $a, b\in \C$ and $a^2+b^2\ne 0$.
	}
\end{rems}

\subsection{Multiplicity and intersection number}\label{subs:MultInter}

Two germs $f_i : (\K^n,0) \to (\K,0)$, with $i=1,2$ and $\K=\C$ or $\R$, are said to be $\mathcal{R}$-\textit{equivalent} if there is a germ of a diffeomorphism $ h:(\K^n,0) \to (\K^n,0)$ such that $f_2 = f_1 \circ h^{-1}$. They are $\mathcal K$-\textit{equivalent} if 
there exists an invertible germ $A : (\K^n,0) \to \K$ (i.e., $A(0)\ne 0$) and a germ of a diffeomorphism $ h$ such that $f_2 (x)=A(x)(f_1 \circ h^{-1})(x)$.
Two germs $f_i : (\K,0) \to (\K^n,0)$, with $i=1,2$, are $\mathcal A$-\textit{equivalent} if there are  germs of diffeomorphisms $ h:(\K,0) \to (\K,0)$ and $ k:(\K^n,0) \to (\K^n,0)$ such that
$f_2= k\circ f_1 \circ h^{-1}$.

If $G$ is one of the above groups ($\mathcal R$, $\mathcal K$, $\mathcal A$) a germ $f$ is $k$-$ G$-{\it determined} if any germ $g$ with $j^kg=j^kf$ is $ G$-equivalent to $f$ (where $j^kf$ is the $k$-jet of $f$ at the origin, that is its Taylor polynomial of degree $k$ without the constant term).   A germ $f$ is said to be $G$-{\it finite} if it is $k$-$ G$-determined for some $k$. There are algebraic conditions determining if a germ is $G$-finite, discussed in \cite{wall} for example. A germ $f:(\C^2,0)\to (\C,0)$ is ${\cal K}$-finite if and only if $f$ has an isolated singularity at $0$, which is equivalent to $f$ having no repeated factors.

Let $\mathcal{O}_n$ denote the ring of germs of holomorphic functions $\C^n,0\to \C$. The order of $\rho=\sum a_kt^k \in  \mathcal{O}_1$,  denoted by $\ord(\rho)$, is the smallest positive integer $k > 0$ for which $a_k\ne 0$, so
$$
\ord(\rho) = \dim_\mathbb{\C} \frac{ \mathcal{O}_1}{\mathcal{O}_1\langle \rho \rangle},
$$
where $\mathcal{O}_1\langle \rho \rangle$ is the ideal in $ \mathcal{O}_1$ generated by $\rho$. Clearly, $\ord(\rho \sigma) = \ord(\rho)+\ord(\sigma)$ for any $\rho, \sigma \in \mathcal{O}_1$. Furthermore, if $\rho(0)=0$ then $\ord(\rho')=\ord(\rho)-1$.

As above, a  germ of a holomorphic function $f : (\C^2,0) \to (\C,0)$ with an isolated singularity defines a germ of a holomorphic curve $(C_f,0)\subset (\C^2,0)$. Decomposing $f$ into irreducible factors in $\mathcal{O}_2$, $f = f_1\cdot f_2 \cdots f_n$ clearly gives $C_f = C_{f_1} \cup \cdots \cup C_{f_n}$. The curves $C_{f_i}$ are called \textit{branches} or \textit{components} of $C_f$. Furthermore, if we write
$$
f= F_r+F_{r+1}+\cdots,
$$
where each $F_i$ is a homogeneous polynomial of degree $i$ in the variables $x$ and $y$ and $F_r \neq 0$, then $r$ is the \textit{multiplicity} of $f$, denoted by $mult(f)$, and $C_{ F_r}$ is the \textit{tangent cone} of $C_f$. We shall also consider parametrisations of curves $\gamma:(\C,0)\to (\C^2,0)$; we will ask that these are ${\cal A}$-finite, which reduces to them being smooth injections. Any such map can be written in the form $(t^m,y(t))$ for some positive integer $m$ with $\ord(y)>m$. If we consider the powers appearing in $y(t)$, the lowest not divisible by $m$ is called the \textit{first Puiseux exponent} of $\gamma$. If $\gamma$ is singular (that is $\gamma'(0) = 0$) and ${\cal A}$-finite then the first Puiseux exponent is finite, since $m>1$ and otherwise $y(t)=y_1(t^m)$ for some $y_1$ and $\gamma$ is not injective. This is clearly an analytic invariant of $\gamma$ and of the corresponding defining equation $f$.

 Given $f, g \in \mathcal{O}_2$ the  \textit{intersection number} of $f$ and $g$, denoted by $m(f,g)$, is the codimension of the ideal $\mathcal{O}_2\langle f,g \rangle$ in $\mathcal{O}_2$, that is, 
 $$
 m(f,g) = \dim_{\C}  \frac{\mathcal{O}_2}{\mathcal{O}_2\langle f,g \rangle}.
 $$ 
 
 More details about the intersection multiplicity/index can be found in \cite{Hefez}. We will use the following properties:
\begin{center}
\begin{tabular}{ll}
	1. $m(f,g)=m(g,f)$;
	&
		4. $m(f,g+fh)=m(f,g)$;
		\\	
	2. $m(uf,vg)=m(f,g)$;
	&
	5. $m(f \circ \Phi,g \circ \Phi) = m(f,g)$; 
	\\
	3. $m(f,gh) = m(f,g)+m(f,h)$;
	&
\end{tabular}
\end{center}
where $f,g,h \in \mathcal{O}_2$, $u,v \in \mathcal{O}_2$ are invertible in $\mathcal{O}_2$ and $\Phi=(\phi_1,\phi_2):(\C^2,0)\to (\C^2,0)$ is a germ of a diffeomorphism. Note that if $f, g$ are irreducible then $m(f,g)$ is non-finite if and only if $f$ and $g$ define the same curve. These are all simple consequences of the definition. In fact,  $m(f,g)$ is the local number of inverses images of a regular value of $(f,g):(\C^2,0)\to (\C^2,0)$; see \cite{Arnold}, Chapter 5. The following results are well-known, the proofs can be found in \cite{Hefez}.

\begin{prop}\label{prop:multOrderParm}
{\rm (1)} Suppose that $C_f$ is an irreducible curve with parametrisation $\gamma:(\C,0)\to (\C^2,0)$. Then $m(f,g)=\ord(g\circ \gamma)$. If $f=0$ is reducible with $r$ branches parametrised by $\gamma_j:(\C,0)\to (\C^2,0), 1\le j\le r$, then $m(f,g)=\sum_{j=1}^r \ord(g\circ \gamma_j)$.

{\rm (2)} Suppose that $C_f$ is an irreducible curve with parametrisation $\gamma(t)=(t^m,y(t))$ with $\ord(y(t))>m$. Then one can write the equation of the curve in the form $y^m+g(x,y)=0$, where $g\in{\cal M}_2^{m+1}$.
	
{\rm (3)} If $f, g$ are irreducible and their tangent cones are transverse, then $m(f,g)=mult(f) \cdot mult(g)$.
	
{\rm (4)} If $f=0$ is smooth and $g$ is irreducible  then $m(f,g)\le \beta(g)$ where $\beta(g)$ is the first Puiseux exponent of $g$.
\end{prop}

%
%
%

\subsection{Geometric invariants of parametrized plane curves}

The  fundamental question is: when a germ of a curve is perturbed how many inflections and vertices emerge? It is important to distinguish if we are defining the curve by a parametrisation or an equation since the set of deformations are even topologically different. The parametrised case is dealt with in \cite{DiasTari}. 
Given a germ of a  parametrised singular curve $\gamma$, one has the expression for its curvature function $\kappa$ and its derivative $\kappa'$  away from the singular point. 
In \cite{DiasTari}, the number of inflections  $I_{\gamma}$ (resp. vertices  $V_{\gamma}$) concentrated at the singular points is defined as the multiplicity of the numerator $i_{\gamma}$ (resp. $v_{\gamma}$) of $\kappa$ (resp. $\kappa'$).

We rewrite here some of the results from \cite{DiasTari}. 
An inflection at a regular point is where the order of contact of the curve and the tangent line is $\ge 3$. If the curve is parametrised by $\gamma(t)=(x(t),y(t)):(\C,0)\to(\C^2,0)$ then its contact with the lines through $0$ is given by the order of $ax(t)+by(t)$, and the condition on contact $\ge 3$ is $ax'(t)+by'(t)=ax''(t)+by''(t)=0$. These linear equations have a solution if and only if the determinant $x'y''-x''y'$ vanishes. We define the number of inflections at the singular point as being the order of the determinant. This agrees with the definition in \cite{DiasTari} as the determinant is precisely  $i_{\gamma}$ and its order is $I_{\gamma}$.

For a vertex we consider contact with circles; those passing through $0$ are given by $A(x^2+y^2)+2Bx+2Cy=0$ and if the curve is parametrised as above we have $\ge 4$ point if and only if at $t=0$
$$
\begin{pmatrix}
(xx'+yy') & x' & y'\\
(xx''+yy''+x'^2+y'^2) & x'' & y''\\
(xx'''+yy''') +3(x'x''+y'y'') & x''' & y'''
\end{pmatrix}
\begin{pmatrix}
A\\B\\C
\end{pmatrix}
\begin{matrix}
\\=\\
\end{matrix}
\begin{pmatrix}
0\\0\\0
\end{pmatrix}
$$

These linear equations have a solution if and only if the (Wronskian) determinant vanishes, that is $v_\gamma=(x'^2 + y'^2)(x'y''' -x'''y') + 3(x'x''+y'y'')(x''y'-x'y'')=0$, and the order of this determinant is defined to be $V_{\gamma}$ (this also agrees with the definition in \cite{DiasTari}). 
Note when $A=0$ we have the lines through $0$.
 
\begin{theo}\label{paramcase}
{\rm (1)} Let $\gamma:(\C,0)\to (\C^2,0)$. There is a unique osculating complex circle or line, that is one with maximal order of contact with $C_{\gamma}$ at $0$. If $f$ is singular this order of contact is finite. The maximal order of contact with a circle 
is denoted by $\lambda(\gamma)$.

{\rm (2)} If $\gamma (t)=(t^m,t^n(a_n+y_1(t))), a_n\ne 0, n>m, y_1 \in {\cal M}_1$, then 
$$I_{\gamma}=m+n-3.$$

{\rm (3)} Let $\gamma$ be as above. If $n\ne 2m$, then 
$$V_{\gamma}=3m+n-6.$$ 
If $n=2m$, then 
$$
V_{\gamma}=3m+n+\ord(a_{2m}t^{2m}(a_{2m}+y_1(t))^2-y_1(t))-6.
$$

{\rm (4)} Let $\gamma$ be as above. Then 
$$V_{\gamma}=I_{\gamma}+\lambda(\gamma)-3.$$
\end{theo}

\begin{proof}
(1)  We claim that if $C_f$ is singular there is unique circle/line with maximum finite order of contact; that is there is a unique osculating circle if we consider lines as circles through infinity.  The circles/lines through $0$ are given by
$A(x^2+y^2)+2Bx+2Cy=0$, where not all of $A, B, C$ are zero. In the singular case suppose that $x(t)=t^m$ and $y(t)=t^n(a_n+y_1(t))$, with $a_n\ne 0$, $n>m$ and $y_1 \in {\cal M}_1$, and set
 $$
 d(t)=A(t^{2m}+t^{2n}(a_n+y_1(t))^2)+2Bt^m+2Ct^n(a_n+y_1(t)).
 $$ 

This has order $m$ unless $B=0$, so the osculating circle/line lies in the pencil given by $B=0$. 
Suppose now that $B=0$.
If $2m<n$ then the order is $2m$ unless $A=0$; so the unique osculating \lq circle\rq\ in this case is the line $y=0$ and the order of contact is $n$. All the tangent circles have contact of order $2m$. 
If $n<2m$ the order is $n$ unless $C=0$ when the osculating circle is $x^2+y^2=0$, and the order is $2m$. 
Finally if $n=2m$, setting $A=1$ 
and taking $C=-1/(2a_{2m})$ then
the order of contact is 
$2m+\ord(a_{2m}t^{2m}(a_{2m}+y_1(t))^2-y_1(t))$. Note $a_{2m}t^{2m}(a_{2m}+y_1(t))^2-y_1(t)$ cannot vanish identically since this would mean that $y_1(t)=y_2(t^m)$ for some $y_2$ and the parametrisation is not reduced. Of course it is classical that when $C_{\gamma}$ is smooth there is a unique osculating circle/line, but then the order of contact may be infinite.

(2) For inflections the lowest order terms appearing in the Wronskian are
$$
\begin{vmatrix}
mt^{m-1}& nt^{n-1}\\
m(m-1)t^{m-2} & n(n-1)t^{n-2}
\end{vmatrix}
$$
which gives $mn(n-m)t^{m+n-3}$ so $I_{\gamma}=m+n-3$.

(3) Suppose that the order of contact with the circle $A(x^2+y^2)+2Bx+2Cy$ is $k=\lambda(\gamma)$, so $A(x(t)^2+y(t)^2)+2Bx(t)+2Cy(t)=2\ell t^k+O(k+1)$ for some $\ell \ne 0$ and we have
{\small
$$
\begin{pmatrix}
(xx'+yy') & x' & y'\\
(xx''+yy''+x'^2+y'^2) & x'' & y''\\
(xx'''+yy''') +3(x'x''+y'y'') & x''' & y'''
\end{pmatrix}
\begin{pmatrix}
A\\B\\C
\end{pmatrix}
\begin{matrix}
\\=\\
\end{matrix}
\begin{pmatrix}
\ell kt^{k-1}+O(k)\\\ell k(k-1)t^{k-2}+O(k-1)\\ \ell k(k-1)(k-2)t^{k-3}+O(k-2)
\end{pmatrix}.
$$
}

We write this as $M(t)D=\alpha(t)$. The adjugate of the above matrix $M(t)$ has for its first row the entries $(x''y'''-x'''y'',x'''y'-x'y''',x'y''-x''y')$ and the lowest (non-vanishing) terms are 
$$
(m(m-1)n(n-1)(n-m)a_nt^{m+n-5}, mn(m-n)(m+n-3)a_nt^{m+n-4},mn(m-n)a_nt^{m+n-3}).
$$

Now $adj(M(t))M(t)=\det M(t)I$ and the order of $\det M(t)$ is $V_{\gamma}$. So the first entry of $adj(M(t))\alpha(t)$ has (potentially) leading term  $t^{m+n+k-6}$ with coefficient
$$
mn(n-m)(k-m)(k-n)k\ell a_n.
$$ 

We need $k\ne m, n$. If $n<2m$ we must have $B=C=0$ and $k=2m$. If $n>2m$ then $B=0$ and $k=2m$.
If $n=2m$ then $B=0$ and we must choose $A, C$ so that the $\ord(d(t))>2m$, that is we have the osculating circle. In all cases $k$ is $\lambda(\gamma)$ and comparing orders $m+n+k-6=V_{\gamma}$, that is $V_{\gamma}=I_{\gamma}+\lambda(\gamma)-3$. 
(Observe that if $A=B=0$, then $k=n$ but the circle degenerates to a line.)

Part (4) follows from (1), (2) and (3).
\end{proof}

\begin{prop}
 Given $\gamma:(\K,0)\to (\K^2,0)$ there is a perturbation $\gamma_1(t)=(x(t)+a_1t+a_2t^2,y(t)+b_1t+b_2t^2), (a,b)$ arbitrarily small, with only simple inflections and vertices.
\end{prop}

\begin{proof}
Writing $X=x(t)+a_1t+a_2t^2$ and $Y=y(t)+b_1t+b_2t^2$ we consider $F(t,a,b)=X'Y''-X''Y'$. Then, at $a_j=b_j=0$,
$$
\partial F/\partial a_j=jt^{j-2}(tY''-(j-1)Y'),\quad \partial F/\partial b_j=jt^{j-2}(tX''-(j-1)X').
$$
So we have a submersion unless $X'(t)=X''(t)=Y'(t)=Y''(t)=0$. These imply $a_1=x'+tx'', a_2=-x''/2, b_1=y'+tx'', b_2=-y''/2$. So for values of $a_j, b_j$ arbitrarily close to $0$, $F$ is a submersion, and we can find nearby $a,b$ with $0$ a regular value of $F_{(a,b)}:\C\to \C$. 

A proof for vertices along the same lines is possible but significantly more complicated. As an alternative we perturb $\gamma$ to obtain an immersion with normal crossings, and then use \cite{Montaldi} perturb this to obtain an immersion which is generic with respect to height and distance-squared functions (suitably interpreted in the complex case).
\end{proof}

It is worth pointing out that except in the regular case the curvature function $x'y''-y'x''$ is {\it never} versally unfolded by varying the parametrisation.

\section{Geometric invariants of plane curves $f=0$} \label{sec:GeometricInvariants}


We deal here with the case when the curve is given by an equation.
Given a germ of a  non-constant holomorphic function $f$, we can calculate the curvature $\kappa$ in (\ref{eq:curvature}) and its derivative at regular points of $C_f$ using the implicit function theorem.
Taking the numerators of those expressions leads to the following  definitions.

\begin{defn} \label{def:IfVf}
	Let $f$ be a holomorphic function.
	
{\rm (1)} We say that $p_0 = (x_0,y_0)$ is an {\it inflection of $C_f$} if $p_0$ is a solution of the equations 
\begin{equation}\label{eq:sis_I}
	\left\{\begin{array}{l}
		f(x,y) = 0,\\
		i_f(x,y) = 0,
	\end{array}\right.
\end{equation}
with 
$$
i_f = f_y^2f_{xx} -2  f_x f_yf_{xy}+ f_x^2f_{yy}.
$$ 

{\rm (2)} We say that $p_0$ is a {\it vertex of $C_f$} when $p_0$ is a solution of the equations
\begin{equation}\label{eq:sis_V}
\left\{\begin{array}{l}
f(x,y) = 0,\\
v_f(x,y) = 0,
\end{array}\right.
\end{equation}
where
$$
\begin{array}{rl}
	v_f = & (f_x^2+f_y^2) \left( f_{x}^3 f_{yyy}-3 f_{x}^2 f_{y} f_{xyy} +3 f_{x} f_{y}^2 f_{xxy} -f_{y}^3 f_{xxx} \right)\\
	& -3((f_x^2- f_y^2)f_{xy} - f_xf_y(f_{xx} -f_{yy}) )(f_y^2f_{xx} -2  f_x f_yf_{xy}+ f_x^2f_{yy})
\end{array}
$$

{\rm (3)} We shall consider, without loss of generality, the point of interest $p_0$ to be the origin and define the number of inflections $I_f$ and vertices $V_f$ of $C_f$ concentrated at the origin by
$$
I_f = m(f,i_f) \quad \textrm{and} \quad V_f = m(f,v_f).
$$
\end{defn}

\begin{rems}
{\rm 
(1) In the case of inflections, if we have a homogeneous polynomial $F$ with $(0,0,1)$ on $F=0$ and $f(x,y)=F(x,y,1)$ then the intersection number of the Hessian of $F$ in the affine chart $z=1$ with $F=0$ at $(0,0,1)$ is $I_f$.   

(2) When finite we prove in Corollary \ref{cor:IVfinite} that $I_f$ and $V_f$ are the number of inflections and vertices that emerge from the origin when deforming $f$. We say that an inflection (resp. vertex) is \textit{ordinary/simple} when $I_f = 1$ (resp. $V_f = 1$). 
}
\end{rems}

\begin{theo}\label{theo:propertiesIf_V_f}
{\rm (1)}  The integer $I_f$ is an affine invariant and $V_f$ is invariant under similarities.

{\rm (2)} At smooth points of $C_f$ the above definitions correspond to the classical notions of inflection and vertex. An inflection {\rm(}resp. vertex{\rm)} is ordinary if and only if $f$ is a submersion at $0$ and $0$ is a simple inflection {\rm(}resp. vertex{\rm)} as in {\rm Definition \ref{def:vert}}. Moreover $I_f+2$ is the order of contact between $C_f$ and its tangent line at $0$.  If the centre of curvature is finite then $V_f$ is the order of contact between $C_f$ and its osculating circle; otherwise the point is an inflection and $V_f= I_f-1$.  

{\rm (3)} Any singular point of $C_f$ is by these definitions an inflection and vertex.

{\rm (4)} If $g=af$ with $a(0)\ne 0$ then $i_g=i_f+\alpha f$ and $v_g=v_f+\beta f$ for some $\alpha, \beta\in {\cal O}_2$.

{\rm (5)} The integers $I_f$ and $V_f$ when finite only depend on a finite jet of $f$.

{\rm (6)} If $f=gh$, and $I_g$ and $I_h$ are finite then $I_f$ is finite and 
$$
I_f=I_g+I_h+6m(g,h).
$$
Similarly, if $V_g$ and $V_h$ are finite then $V_f$ is finite and 
 $$
 V_f=V_g+V_h+12 m(g,h).
 $$

{\rm (7)}  The invariant $I_f$ is infinite if and only if $(C_f,0)$ contains the germ of a line; $V_f$ is infinite if and only if $(C_f,0)$ contains the germ of a complex circle or line.
\end{theo}

\begin{proof}
(1), (3) and the first two parts of (2) are obvious. For the third part at a smooth point we may write $f(x,y)=y+g(x)$ and  so $i_f=g_{xx}$ and $I_f=\dim_{\C}{\cal O}_1/{\cal O}_1\langle g_{xx}\rangle$ which is $2$ less than the order of contact between the curve and its tangent. The vertex case follows similarly.

(4) is clear and for (5) we provide a proof in the inflection case only. Similar calculations yield the results for vertices. Let ${\cal M}_2$ denote the maximal ideal in ${\cal O}_2$. The regular case is similar to the singular case, where we can prove something stronger, namely if $f$ is singular and $I_f=k$ then $I_{f+h}=I_f$ for any $h\in{\cal M}^{k+1}$. Suppose then that $f$ is singular and $I_f=k$.  It follows from Nakayama's Lemma (\cite{wall}, Lemma 1.4 (ii)),  that ${\cal M}_2^{k}\subset {\cal O}_2\langle f,i_f\rangle$. Replacing $f$ by $f+h$ where $h\in {\cal M}_2^{k+1}$ note that $i_{f+h}=i_f+H$ where $H\in {\cal M}_2^{k+1}$, this uses the fact that $f$ is singular. So 
$$
{\cal M}_2^k\subset {\cal O}_2\langle f,i_f\rangle={\cal O}_2\langle f+h,i_{f+h}\rangle+ {\cal M}_2^{k+1}
$$
and by Nakayama's Lemma ${\cal M}_2^{k}\subset {\cal O}_2\langle f+h,i_{f+h}\rangle$. It follows that 
$$
\dim_{\C} \frac{{\cal O}_2}{{\cal O}_2\langle f, i_f\rangle}=\dim_{\C} \frac{{\cal O}_2}{{\cal O}_2\langle f, i_f\rangle+ {\cal M}_2^{k+1}}=\dim_{\C}\frac{{\cal O}_2}{{\cal O}_2\langle f+h, i_{f+h}\rangle}.
$$

(6) We first observe that
$
i_{f} = g^3 i_{h}+h^3 i_{g}+g h r_1,
$
where $r_1$ depends on $g$, $h$ and their derivatives. Using the properties of the intersection number, we get
	$$\begin{array}{rcl}
		I_f & = & m(gh,g^3 i_{h}+h^3 i_{g}+g h r_1)=m(g,h^3 i_{g})+m(h,g^3 i_{h})\\
		    & = & m(g,i_g)+m(h,i_h)+m(g,h^3)+m(h,g^3)=I_g+I_h+6m(g,h).
	\end{array}
$$

(7) If $I_f$ is infinite then $C_f$ and $\{i_f=0\}$ share a component. So at any smooth point $p$ of this intersection $I_f$ is infinite, and (2) shows that locally $C_f$ contains a line, and the conclusion follows. A similar argument establishes the vertex case.
 \end{proof}
 
\begin{theo}\label{teo:calc_fat_irred}
	If $f=f_1 \dots f_n$, then:
$$
\begin{array}{rcl}
	I_f &=& \sum_{i}^{n} I_{f_i}+6\sum_{i<j} m(f_i,f_j),\\
	V_f &=& \sum_{i}^{n} V_{f_i}+12\sum_{i<j} m(f_i,f_j).
\end{array}
$$

\end{theo}

\begin{proof}
The proof follows by induction using Theorem \ref{theo:propertiesIf_V_f} (7)  when the $I_{f_i}, V_{f_i}$ are finite, and is trivial if any are infinite.
\end{proof}

Clearly we are interested in the pairs $C_f$, $\{i_f=0\}$ (resp. $C_f$, $\{v_f=0\}$). 
There are classifications of pairs $(f,g)$ of this form where $f$ and $g$ are \lq independent\rq\ due to Goryunov, \cite{gory}.
His equivalence is a subgroup ${\cal G}$ of the corresponding ${\cal K}$ group, given by changes of co-ordinates in the source, with the addition that $(f,g)$ is equivalent to $(\alpha f,\beta f+\gamma g)$ where $\alpha, \beta, \gamma$ are function germs and $\alpha (0)\gamma (0)\ne 0$.
We are interested in the case $g=i_f$ or $g=v_f$. Note that the ${\cal G}$-type of the germs $(f,i_f), (f,v_f)$ do not depend on the choice of defining equation for the curve by Theorem \ref{theo:propertiesIf_V_f}, so their classes under ${\cal G}$-equivalence and indeed ${\cal K}$-equivalence provide additional invariants.

Before considering $(f,i_f)$ and $(f,v_f)$ we need the following result.

\begin{prop}\label{prop:KGfinite}
{\rm (1)} Let $F=(f,g):(\C^2,0)\to (\C^2,0)$ be a finite holomorphic germ, that is $(F^{-1}(0),0)=\{0\}$. Then $(c,0)$ is a regular value of $F$ for $c\ne 0$ sufficiently small if and only if $g$ is ${\cal K}$-finite.
	
{\rm (2)} A germ $F=(f,g):(\C^2,0)\to (\C^2,0)$ is ${\cal G}$-finite if and only if $f$ and $(f,g)$ are ${\cal K}$-finite. Note $g$ need not be ${\cal K}$-finite {\rm(}for example take $F=(x,y^2))$.
\end{prop}

\begin{proof}
(1) Since $F$ is holomorphic its critical set $\Sigma$ will be an analytic variety. If it has dimension $2$ then clearly $F$ is constant, contradicting finiteness. So we may suppose that $\dim \Sigma=1$. The points $(c,0)$ are not regular values of $F$ for $c\ne 0$ sufficiently small only if for some irreducible component $\Sigma_1$ of $\Sigma$ we have $F(\Sigma_1)$ the $u$-axis in the $(u,v)$ target space. If $F|\Sigma_1$ is not an immersion for points $(x,y)\in \Sigma_1 $ sufficiently close to $(0,0)$ then locally $F(\Sigma_1)=(0,0)$ and again we have a contradiction. So away from $(0,0)$ locally the image of the tangent space under $dF(x,y), (x,y)\in \Sigma_1$ is the $u$-axis, which means that $dg$ is singular along $\Sigma_1$ and $g$ is not ${\cal K}$-finite. Conversely if $g$ is not ${\cal K}$-finite then $g$ is singular along a component of $g=0$ and clearly $(c,0)$ is not a regular value of $F$ for $c$ small.

(2) We follow Wall's account as in \cite{wall}, p. 491, of Gaffney's geometric characterisation of finite determinacy. The ${\cal G}_e$-tangent space is the ${\cal O}(x,y)$-module
$$
T_e{\cal G}(f,g)={\cal O}(x,y)\{(f_x,g_x),(f_y,g_y), fe_1, fe_2,ge_2\},
$$
and by definition $(f,g)$ is ${\cal G}_e$-stable if and only if  $T_e{\cal G}(f,g)={\cal O}(x,y)^2$. Clearly if $f(x,y)\ne 0$ we have local stability at $(x,y)$ and similarly if $(x,y)$ is a regular point of $(f,g)$. Indeed it is easy to see that the instability locus is given by 
$$
(\{f=g=0\}\cap \Sigma(f,g))\cup (\{f=0\}\cap(\{f_x=f_y=0\}).
$$  

As explained in \cite{wall} from the general theory of sheaves we can deduce that $(f,g)$ is finitely ${\cal G}$-determined if and only it is ${\cal G}$-stable off $0\in \C^2$, and the result follows.
\end{proof}

\begin{cor}\label{cor:IVfinite}
{\rm (1)} If $I_f$ {\rm(}resp. $V_f${\rm)} is finite, then $(f,i_f)$ {\rm(}resp. $(f,v_f)${\rm)}  is ${\cal G}$-finite.

{\rm (2)} If $i_f$ {\rm(}resp. $v_f${\rm)} is ${\cal K}$-finite and $I_f$ {\rm(}resp. $V_f${\rm)}  is finite, then for some small regular value $c$ of $f$ we find that $f^{-1}(c)$ has only simple inflections {\rm(}resp. vertices{\rm)} and of course their number is $I_f$ {\rm(}resp. $V_f${\rm)} .
\end{cor}

\begin{proof}
 (1) Take $F=(f,i_f)$ (resp. $(f,v_f)$); since these are finite mappings they are ${\cal K}$-finite, and since $f$ is also ${\cal K}$-finite the result follows from Proposition \ref{prop:KGfinite}.

For (2)  since $I_f$ (resp. $V_f$) are finite $F$ is finite, and the result now follows from Proposition \ref{prop:KGfinite}(1).
\end{proof}

So far we have focused on the complex case; we pause to state some results which also hold over the reals, noting that there is an obvious real analogue of ${\cal G}$-equivalence. We recall the following notions of Tougeron. Let $\K=\R$ or $\C$ and $J^k(n,p)$ denote the jet space of polynomial mapping $\K^n\to \K^p$ of degree $d$, with $ 1\le d \le k$. There are natural projections $\pi_{k+1}:J^{k+1}(n,p)\to J^k(n,p)$. If $A_k\subset J^k(n,p)$ are algebraic sets with $A_{k+1}\subset \pi_{k+1}^{-1}A_k$ the set $A$ of germs $f$ with $j^kf\in A_k$ for all $k$ is said to be proalgebraic. Clearly ${\rm codim} A_k\le {\rm codim} A_{k+1}$, write ${\rm codim} A=\lim_{k\to \infty}{\rm codim} A_k$. A property of map-germs is said to {\it hold in general} if it holds for all germs except those in a proalgebraic set of infinite codimension.
A straightforward argument \cite{wall}, p. 513, shows that ${\rm codim} A=\infty$ if and only if for every $z\in J^k(n,p)$ we can find $f\notin A$ with $j^kf=z$.

We have the following result whose proof we give in the Appendix.

\begin{theo}\label{theo:GenericIV}
{\rm (1)} In general smooth germs $f:(\K^2,0)\to(\K,0)$ yield  maps $(f, i_f):(\K^2,0)\to (\K^2,0)$, $(f, v_f):\K^2,0\to \K^2,0$ which are ${\cal G}$-finite.

{\rm (2)} In general smooth germs $f:(\K^2,0)\to(\K,0)$ have $I_f, V_f$ finite.

{\rm (3)} In general  smooth germs $f:(\K^2,0)\to (\K,0)$ have $i_f, v_f$ ${\cal K}$-finite.

{\rm (4)} Given any ${\cal K}$-finite germ $f$ there is a ${\cal K}$-equivalent germ $g$ with $(g,i_g), (g,v_g)\ {\cal G}$-finite, $I_g, V_g$ finite and $i_g, v_g\ {\cal K}$-finite. 
\end{theo}

\begin{rems}
	{\rm
(1)	The results in Theorem \ref{theo:GenericIV} show that for any ${\cal K}$-class $f$ almost all representatives (the complement has infinite codimension, that is can be avoided in any finite dimensional families) have the property that $I_f$ and $V_f$ are finite and $(f,i_f), (f,v_f)$ are ${\cal G}$-finite. Moreover there is a nearby non-singular fibre $f=c$ with only simple inflections and vertices. In particular there is a minimum value for $I_f, V_f$. 

(2) It follows from Theorem \ref{theo:propertiesIf_V_f}, when $\K=\C$, if $C_f$ has no smooth components $I_f, V_f$ are always finite. In that case we prove in Theorem \ref{theo:bounds} that they can only take on a finite number of values. If $C_f$ has a smooth component then $I_f, V_f$ can be arbitrarily large.

(3) Our focus below will be on computing the integers $I_f, V_f$, but the ${\cal G}$-type of $(v,i_f), (f, v_f)$, and their ${\cal G}$-codimensions are potentially interesting invariants worth investigating.
}
\end{rems}

\subsection{Irreducible curves}\label{subs:irreducible}

By Theorem \ref{teo:calc_fat_irred}, to determine $I_f$ and $V_f$ for $f\in {\mathcal O}_2$, it is enough
to compute these numbers for each irreducible component of $f$. For this reason 
we consider here the case of irreducible curves, so we assume $f$ to be irreducible with multiplicity $m$ and 
 consider a parametrisation of $C_f$ in the form
\begin{equation}\label{eq:puiseux}
	\gamma(t) = (t^m,y(t)),
\end{equation}
with $y(t) \in \mathcal{O}_1$ and $\ord(y) > m$, so that
\begin{equation}\label{eq:Hefez}
	m(f,g) = \dim_\mathbb{C} \frac{\mathcal{O}_2}{\langle f,g \rangle} = \dim_\mathbb{C} \frac{\mathcal{O}_1}{\langle g\circ \gamma \rangle} = \ord(g \circ \gamma)
\end{equation}
for any $g \in \mathcal{O}_2$ (see Proposition \ref{prop:multOrderParm}).

The integers $I_\gamma$ and $V_\gamma$  represent the number of inflections and vertices that arise when deforming $\gamma$, and if $\gamma$ is a parametrization of $C_f$, then it is natural to expect that $I_\gamma \leq I_f$ and $V_\gamma \leq V_f$ because deformations of $\gamma$ arise from special 
deformations of $f$. The following result relates $I_\gamma$ and $I_f$, and $V_\gamma$ and $V_f$, 
where $\mu(f)=m(f_x,f_y)$ is the Milnor number of the germ $f$.

\begin{theo}\label{teo:princ}
Let $f : (\C^2,0) \to (\C,0)$ be a germ of an irreducible holomorphic function 
and $\gamma : (\C,0) \to (\C^2,0)$ be a parametrization of the curve $C_f$. Then

{\rm (1)}  {\rm (\cite{wallflat}, Proposition 4.1)} $I_f = I_\gamma + 3 \mu (f).$

{\rm (2)} $V_f = V_\gamma + 6 \mu (f).$

{\rm (3)} Let $f=f_1 \cdots f_n$, with $f_i$ irreducible and $\gamma_i$ the corresponding parametrisation. Then 
$$
\begin{array}{c}
I_f=\sum_{i=1}^n I_{\gamma_i}+3(\mu(f)+n-1), \\
V_f=\sum_{i=1}^n V_{\gamma_i}+6(\mu(f)+n-1).
\end{array}
$$
\end{theo}

\begin{proof} For (1) we give an alternative proof to that in  \cite{wallflat}. 
	If $\gamma(t)=\left(x(t),y(t)\right)$ as above, then $f(x(t),y(t)) \equiv 0$, so
	\begin{equation}\label{eq:der_F_t}
		\begin{array}{rcc}
			f_x x'+f_y y' & = & 0,\vspace{0.2cm}\\
			f_{x} x''+f_y y''+f_{xx} (x')^2+2 f_{xy} x' y'+f_{yy} (y')^2& = & 0,\\
			x^{(3)} f_{x}+y^{(3)} f_{y}+(x')^3 f_{xxx}+3 (x')^2 y' f_{xxy}+3 x' (y')^2 f_{xyy}+(y')^3 f_{yyy} +&&\vspace{0.2cm} \\
			3 \left(x'' y' f_{xy}+x' y'' f_{xy}+x' x'' f_{xx}+y' y'' f_{yy}\right) & = & 0,
		\end{array}
	\end{equation}
	where the derivatives of $f$ are evaluated at $\gamma(t)$ and the derivatives of $x$ and $y$ are evaluated at $t$. It follows from the first equation that there exists $u: (\C,0) \to \C$, with $u(0) \neq 0$ and an integer $k \geq 0$ such that
	\begin{equation}\label{eq:lambda}
		f_x(x(t),y(t)) = u(t) t^k y'(t) \quad \textrm{and} \quad f_y(x(t),y(t)) = -u(t) t^k x'(t).
	\end{equation}

Now $f(0,y)=y^mf_1(y)$, $f_1(0)\ne 0$, for $y\ne 0$ small, so it follows from Teissier's Lemma (see \cite{Ploski} for more details) that
	\begin{equation}\label{eq:teissier}
		m(f(x,y),f_y(x,y)) = \mu(f)+m(f(0,y),x)-1=\mu(f)+m-1.
	\end{equation}

	We conclude from (\ref{eq:Hefez}), (\ref{eq:teissier}) and (\ref{eq:lambda}) that
	$$
	\ord(f_y(x(t),y(t))) = \mu(f)+m-1=k+m-1,
	$$
so $k=\mu(f)$.
	Multiplying the second equation in (\ref{eq:der_F_t}) by $u(t)^2 t^{2k}$ and using (\ref{eq:lambda}), we have
	$$\left(f_{xx} f_y^2-2 f_{xy} f_x f_y+f_{yy} f_x^2\right)(x(t),y(t)) = -u(t)^3 t^{3k} (x''(t) y'(t)-y''(t)x'(t)),$$
	that is, $i_f(\gamma(t)) = - u(t)^3 t^{3k} i_\gamma(t)$. Therefore,
	\begin{equation}\label{eq:If}
		\begin{array}{rcl}
		I_f & = & m(f,i_f)= \ord(i_f \circ \gamma)=\ord(- u(t)^3 t^{3k} i_\gamma(t))\\
		& = &  \ord (i_\gamma)+3k =I_\gamma+3k=I_{\gamma}+3\mu(f).
		\end{array}
	\end{equation}

	Multiplying the third equation of (\ref{eq:der_F_t}) by $(f_x^2+f_y^2) u(t)^3 t^{3k}$ and using (\ref{eq:lambda}) and the second equation in (\ref{eq:der_F_t}), it follows that $v_f(\gamma(t)) = - u(t)^6 t^{6k} v_\gamma(t)$. With calculations similar to those in (\ref{eq:If}), we get 
	$V_f =V_\gamma+ 6 k=V_\gamma+ 6\mu(f)$. 
	
The formulae in (3) follow from (1) and (2) above, Theorem 3.4 and Theorem 6.5.1 of \cite{wall2} which shows that $\mu(f)=\sum_i\mu(f_i)+2\sum_{i<j} m(f_i,f_j)-n+1$.
\end{proof}

Suppose given a singular irreducible germ $f$ parametrised by $(x(t),y(t))$, and let $\lambda(f)$ denote the maximal contact between $C_f$ and circles through $0$, as in \mbox{Theorem \ref{paramcase}}. Note that  if $f$ is singular, $\lambda (f)\le \beta(f)$  from {\rm Proposition \ref{prop:multOrderParm}.

\begin{theo}\label{teo:diastari}
If $C_f$ is a germ of a singular irreducible curve then
	$$
	V_f = I_f+3 \mu(f)+\lambda (f).
	$$ 
\end{theo}

\begin{proof}
This follows from Theorem \ref{teo:princ} and Theorem \ref{paramcase}.
\end{proof}

\begin{exs}\label{Ex:A1}
{\rm 

1.	{\it $A_1$-singularity}. 
Consider a curve $C_f$ with a Morse singularity ($A_1$). We can write $f=gh$ with $C_g, C_h$ regular and transverse. By Theorem \ref{theo:propertiesIf_V_f}(7) $I_f=I_g+I_h+6$ and $V_f=V_g+V_h+12$, so if $f=0, g=0$ do not have an inflection or vertex at $0$ then $I_f=6, V_f=12$. 
This result could be used to recover Theorem \ref{teo:princ}. 
For if we have an irreducible singularity $f=0$ we first perturb the parametrisation to get a curve $C$ with only ordinary double points and  simple inflections and vertices on the immersed curve with none at the double points. The number of double points (i.e., $A_1$-singularities) is $\delta(f)=\frac{1}{2}\mu(f)$ (\cite{milnor}). From the calculations above if we then perturb $C$ to get a smooth curve $C'$ then $C'$ has $I_{\gamma}+6\delta$ inflections and $V_{\gamma}+12\delta$ vertices, that is $I_f=I_{\gamma}+3\mu(f)$ and $V_f=V_{\gamma}+6\mu(f)$ vertices.

\medskip
2. {\it The Klein cubic}. 
An example where we can see the inflections being absorbed into a double points is provided by a deformation of a non-singular cubic curve. Consider a smooth cubic; its Hessian is also a cubic, which can only meet the cubic in finitely many points - else the cubic contains a line. So there will be $9$ inflections with multiplicities. It is not hard to see that for any nonsingular cubic curve there are $9$ distinct inflections, the line joining any two contains a third. Consider now a nodal cubic: suppose the node is at $(0,0,1)$ so we can reduce to $zxy -C(x,y)$, C cubic, and then a change of co-ordinates $z\mapsto z+\alpha x+\beta y$ and scaling reduces  to $x^3+y^3-xyz=0$. The Hessian is $3(x^3+y^3)+xyz$ and there are only three inflections; also collinear. So the double point has absorbed $6$ inflections. We can see how when we consider the family
$
x^3+y^3+t^3z^3-xyz.
$ 
This is almost in the usual `Steiner' form: set $Z=tz$ we obtain $x^3+y^3+Z^3-t^{-1}xyZ$. We get $9$ inflections at $(x,y,Z)$:
$$
\begin{array}{l}
(0,-1,1), (0,-\omega,1), (0,-\omega^2,1)\\
(1,0,-1), (1,0,-\omega), (1,0,-\omega^2),\\
(-1,1,0), (-\omega,1,0), (-\omega^2,1,0),
\end{array}
$$
so for $(x,y,z)$
$$
\begin{array}{l}
(0,-t,1), (0,-\omega t,1), (0,-\omega^2t,1),
\\
(t,0,-1), (t,0,-\omega), (t,0,-\omega^2),
\\
(-1,1,0), (-\omega,1,0), (-\omega^2,1,0).
\end{array}
$$
The first $2$ sets of points give the $6$ inflections emerging from $(0,0,1)$.
}
\end{exs}

\section{Bounds for $I_f, V_f$ in the irreducible cases} \label{sec:exam}

	We consider germs of irreducible singular curves.
	Suppose that the  curve $C_f$ is parametrised as $\gamma(t) = \left(t^m,\sum_{i=n}^{\infty} a_i t^i \right),$
	with $n>m$ and $a_n\ne 0$; we showed in Theorem \ref{paramcase} that $I_\gamma = m+n-3$ for all $m, n$ and $V_\gamma = 3m+n-6$ if $n\ne 2m$.
	When  $n=2m$, if we write $y(t)=t^{2m}(a_{2m}+y_1(t))$, with $y_1\in{\cal M}_1$ and $a_{2m}\ne 0$, then 
	$$
	V_{\gamma}=3m+n+\ord(a_{2m}t^{2m}(a_{2m}+y_1(t))^2-y_1(t))-6.
	$$
	
	Consider now the case $n=2m$; the formula above gives a perfectly good algorithm for computing $V_{\gamma}$ from the terms in the parametrisation. In this case we have $\gamma(t) = \left(t^m,\sum_{i=2m}^{\infty} a_i t^i \right)$. If $n_1$ is the exponent of the next non-zero term beyond $t^{2m}$ it follows from the expression above that
	$$
	V_\gamma = \left\{ \begin{array}{ll}
		3m+n_1-6 & \textrm{if } n_1 < 4m \,\\
		7m-6 & \textrm{if } n_1 = 4m \textrm{ and } a_{4m}-a_{2m}^3 \neq 0, \textrm{ or } n_1 > 4m.\\
	\end{array}\right.
	$$
	
	The condition $a_{4m}-a_{2m}^3=0$ emerges because $(T,a_{2m}T^2-a_{2m}T^3)$ is the initial part of a parametrisation of a circle centred at $(0,1/a_{2m})$; we have $T=t^m$. A more explicit expression for $\lambda(\gamma)$ is given in the next Lemma. The conditions on the $a_j$ that emerge do so for the same reason.
	
	\begin{lem} \label{lambda explicit}
		More explicitly 
		$$
		\lambda(\gamma) = \left\{ \begin{array}{ll}
			2lm& \textrm{if } a_{(2j+1)m}=0,j=1,..,l-1, 2lm<\beta(\gamma),\\
			& a_{(2j)m}=\alpha_{2j}a_{2m}^{2j-1},j=2,..,l-1,\, a_{(2l)m}\ne \alpha_{2l}a_{2m}^{2l-1}\vspace{0.2cm}\\
			(2l+1)m & \textrm{if } a_{(2j+1)m}=0,j=1,..,l-1, a_{(2l+1)m}\ne 0, (2l+1)m <\beta(\gamma),\\
			& a_{(2j)m}=\alpha_{2j}a_{2m}^{2j-1},i=2,..,l\vspace{0.2cm}\\
			\beta(\gamma)& \textrm{otherwise}
		\end{array}\right.
		$$
		where $\alpha_{(2j+1)m}=0$ for all $j$,  $\alpha_{2}=\alpha_{4}=1,$ and
		$\alpha_{2j}=\alpha_{j}^2+2\sum_{p+q=j, p<q}\alpha_{2p}\alpha_{2q}.$
	\end{lem}
	
	Using the formula $V_{\gamma}=I_{\gamma}+\lambda(\gamma)-3$ in Theorem \ref{paramcase} and the fact that $I_{\gamma}=5m-3$ 
	we get 
	{\footnotesize
		\begin{equation}
			\label{eq:V_gammaExcep}
			V_{\gamma} =  \left\{ \begin{array}{ll}
				(2l+5)m-6& \textrm{if } a_{(2j+1)m}=0,j=1,..,l-1, 2lm<\beta_1(\gamma),\\
				& a_{(2j)m}=\alpha_{2j}a_{2m}^{2j-1},j=2,..,l-1,\, a_{(2l)m}\ne \alpha_{2l}a_{2m}^{2l-1}\vspace{0.2cm}\\
				(2l+6)m-6& \textrm{if } a_{(2j+1)m}=0,j=1,..,l-1,a_{(2l+1)m}\ne 0, (2l+1)m<\beta_1(\gamma),\\
				& a_{(2j)m}=\alpha_{2j}a_{2m}^{2j-1},j=2,..,l\vspace{0.2cm}\\
				5m+\beta(\gamma)-6& \textrm{otherwise}
			\end{array}
			\right.
		\end{equation}
	}	
	
	We can now be more specific about the genericity result in Theorem \ref{theo:GenericIV}. Indeed, we have the following about the finiteness of $I_f$ and $V_f$.
	
	\begin{theo}\label{theo:bounds}
		{\rm (1)} If $f$ is irreducible and singular,
		then $I_f$ can take on the values $3\mu(f)+(j+1)m-3, jm<\beta(f)$ and $3\mu(f)+m+\beta(f)-3$. In particular, $I_f\le  3\mu(f)+m+\beta(f)-3$.
		
		{\rm (2)} Suppose that $f$ is ${\cal K}$-finite. Then  
		$I_f$ is bounded if and only if $C_f$ has no components which are (affine) lines. If $C_f$ has no nonsingular components, then 
		$\{I_{f'}: {f'}\sim_ {\cal K}f \}$ is bounded.
		
		{\rm (3)} If $f$ is irreducible and singular then $V_f\le 6\mu(f)+m+2\beta(f)-6.$
		
		{\rm (4)} Suppose that $f$ is ${\cal K}$-finite. Then  
		$V_f$ is bounded if and only if $C_f$ has no components which are parts of a complex circle or line. 
		If $C_f$ has no nonsingular components, then $\{V_{f'}: {f'}\sim_ {\cal K}f \}$ is bounded.
	\end{theo}
	
	\begin{proof}
		(1) is immediate from above, and (2) follows from (1) and Theorem \ref{theo:propertiesIf_V_f}. Indeed, suppose that $f$ is reducible and write $f=f_1\ldots f_n$, with the $f_i$'s irreducible.  
		As $f$ is $\mathcal K$-finite, the  $m(f_i,f_j)$ are all finite, so $I_f$ is not finite if and only if  $I_{f_i}=\infty$ for some $i$. But by (1) this 
		can only happen when $f_i$ is a smooth component of $f$. Writing $f_i=y-h(x)$ the latter means that $(f_{i})_{xx}$ is identically zero and we have an affine line.
		
		(3) Suppose that $C_f$ is irreducible and not smooth; then it has finite order of contact with an osculating circle, indeed $\lambda(\gamma)\le \beta(f)$. Since $V_f=V_{\gamma}+6\mu(f)=I_{\gamma}+\lambda(\gamma)-3+6\mu(f)=I_f+\lambda(\gamma)+3\mu(f)-3$, it follows that $V_f$ is bounded by $6\mu(f)+m+2\beta(f)-6$.
		
		(4) Follows similarly to (2).
\end{proof}

It is not hard to see that given a germ $f:(\C^2,0)\to (\C,0)$ the generic values of $I_f, V_f$ are the minimal ones. This can be made precise by working in a suitably large jet-space $J^k(2,1)$ using Theorem \ref{theo:propertiesIf_V_f}(5); the details are left to the reader.  These minimal values then are of particular interest.

\begin{prop}
{\rm (1) Irreducible case:} if $f=0$ is smooth then clearly the minimal values of $I_f$ and  $V_f$ are zero. If $f$ is singular, 
then the minimal value of $I_f$ is $3\mu+m+\beta(f)-3$ if $\beta(f)<2m$, or $3(\mu+m-1)$ otherwise, and the minimal value of 
$V_f$ is $6\mu+3m+\beta(f)-6$ if $\beta(f)<4m$ or $6\mu +7m-6$ otherwise.  
	
{\rm (2) General case:} denoting $\min I_f, \min V_f$ for these minimal values and writing $f=f_1\ldots f_n$, $f_i$ irreducible, then
$$\begin{array}{rcl}
	\min I_f&= &\sum_{i=1}^n \min I_{f_i}+6\sum_{i<j}m(f_i,f_j), \\
	\min V_f&=&\sum_{i=1}^n \min V_{f_i}+12\sum_{i<j}m(f_i,f_j).
\end{array}	
	$$
\end{prop}

\begin{proof}
	(1) It follows from Theorem \ref{teo:princ} that in the irreducible case the minimum values for $I_f, V_f$ correspond to the minimal values of $I_\gamma, V_\gamma$. For inflections there are two cases to consider: if the first Puiseux exponent $\beta(f)$ of $\gamma$ is $< 2m$ then $I_\gamma=m+\beta(f)-3$. If $\beta(f)>2m$ then we can introduce a $t^{2m}$ term in the second component without changing the analytic type and $I_\gamma=3m-3$.
	
	For vertices if $\beta(f)<2m$ then $V_\gamma=3m+\beta(f)-6$. If $\beta(f)>2m$ then again we can introduce a $t^{2m}$ term in the second component and again $V_\gamma=3m+\beta(f)-6$ if $\beta(f)<4m$. If $\beta(f)>4m$ then choosing $a_{4m}=0, a_{2m}=1$ we see that $V_\gamma=7m-6$. 
	
	(2) We need to show that given an analytic type $f$ we can choose a representative with each irreducible factor $f_i$ having the minimal value for $I_{f_i}, V_{f_i}$. Again working in a suitably large jet-space $J^k(2,1)$ one can check, for each $i$, that the set of algebraic diffeomorphism germs $\phi:(\C^2,0)\to (\C^2,0)$ of degree $\le k$ with $\phi\circ \gamma_i$ not having the minimal value is a proper algebraic subset of this open affine set. That it is algebraic follows from the fact that these numbers occur as orders of functions of $1$-variable. So we need to prove that it is proper, that is find $\phi$ with $\phi\circ \gamma_i$ giving the minimal value, which is not difficult. For example for inflections if $\beta(f)<m$ there is nothing to prove. If $\beta(f)>m$ then either $a_{2m}\ne 0$ and we choose $\phi$ to be the identity, or choose $\phi$ to be $(x,y)\mapsto (x,y-x^2)$.  Now the complement of a finite number of proper algebraic sets is dense, so we can choose a diffeomorphism $\phi$ with $\phi\circ f_i$ having the minimal value of $I_{f_i}, V_{f_i}, 1\le i\le n$ and the result follows. 
\end{proof}

\subsection{Simple singularities}
The $\mathcal K$-simple (which are also the $\mathcal R$-simple) singularities 
of germs of functions are classified by Arnold and, when $n = 2$, are $\mathcal{K}$-equivalent to the following normal forms (where $\pm$ is for the case $\K=\R$ and should be replace by $+$ when $\K=\C$):
$$\begin{array}{ll}
	A_k:& \pm(x^2 \pm y^{k+1}), k \geq 1, \\ 
	D_k:& x^2 y\pm y^{k+1}, k \geq 4, \\
	E_6:& x^3+y^4, \\
	E_7:&x^3+x y^3, \\
	E_8:&x^3+y^5.
\end{array}
$$

A fundamental result of Milnor shows that $\mu=2\delta-r+1$ where $r$ is the number of (complex) irreducible components of $f=0$ and $\delta$ is the number of (complex) double points in a generic deformation of the multi-parametrisation; see \cite{milnor}. This result was quoted in \S \ref{sec:GeometricInvariants} when $r=1$.

Using these results and the formulae in Theorem \ref{teo:calc_fat_irred} and Theorem \ref{teo:princ} it is not hard to calculate $I_f$ and $V_f$ for curves with a singularity of type $A_k,D_k,E_6,E_7,E_8$. The results are given below. We illustrate the calculations in a couple of cases.

\medskip
$\bullet$ {\bf $A_{2k}$-singularity}
\medskip

For the $A_l$-singularities, the curve has a single branch if $l=2k$ is even and two regular branches when $l=2k+1$ is odd. We start with the former.
 
When $k=1$, the curve has a cusp singularity 
and can be parametrized by $\gamma(t) = (t^2,a_3 t^3+O(t^4))$, with $a_3 \neq 0$. We have $I_\gamma = 2$ and $V_\gamma = 3$ (\cite{DiasTari}) so by Theorem \ref{teo:princ} we 
have $I_f = 8$ and $V_f = 15$.

Suppose that  $k \geq 2$ and write
$$
\gamma(t) = \left( t^2,a_4t^4+a_6t^6+\ldots+a_{2k}t^{2k}+a_{2k+1}t^{2k+1}+O(2k+2) \right),
$$
with $a_{2k+1}\ne 0$. Then by \cite{DiasTari}, see also Theorem \ref{paramcase},
$$
I_\gamma = 
\left\{ 
\begin{array}{ll}
2j-1& \textrm{if } a_{2p}=0, 3\le p\le j-1, a_{2j}\ne 0, 2\le j\le k,\\
2k & \textrm{if } a_{2p}=0, 2\le p\le k
\end{array}\right.
$$ 
and $I_f=I_{\gamma}+6k$ (as $\mu(f)=2k$), with $I_{\gamma}$ as above. Therefore, 
$I_f$ can have one of the values $6k+2j-1$ for $2\le j\le k$ or $8k$. Examples of defining equations with these values
are $(y-x^{2j})^2-x^{2k+1}, 2\le j\le k,$ and $y^2-x^{2k+1}$.

On the other hand
$$
V_{\gamma}=\left\{ 
\begin{array}{ll}
	2j& \textrm{if } a_{2p}=0, 3\le p\le j-1, a_{2j}\ne 0, 2\le j\le k,\\
	2k +1 & \textrm{if } a_{2p}=0, 2\le p\le k
	\end{array}\right.
$$

We need above  $a_8-a_4^3\ne 0$ when $a_4a_8\ne 0$ and $a_6=0$.
We have $V_f=V_{\gamma}+12k$, so 
$V_f$ can have one of the values $12k+2j$ for $2\le j\le k$, or $14k+1$. (The increases are similar to those of $I_f$, it jumps value by $2$ until the last jump which is by $1$.)

When $a_8-a_4^3=0$, $a_4a_8\ne 0$ and $a_6=0$, we get 
\begin{equation*}
V_{\gamma} =  \left\{ \begin{array}{ll}
	4l+4& \textrm{if } a_{2(2j+1)}=0,j=1,..,j-1, 2l<k,\\
	& a_{4j}=\alpha_{2j}a_{4}^{2j-1},i=2,..,l-1,\, a_{4l}\ne \alpha_{2l}a_{4}^{2l-1}\\
	4l+6& \textrm{if } a_{2(2j+1)}=0,j=1,..,l-1,0\ne a_{2(2l+1)}, 2l <k,\\
	& a_{4j}=\alpha_{2i}a_{4}^{2j-1},j=2,..,l\\
	2k+5& \textrm{otherwise}
\end{array}\right.
\end{equation*}

\medskip
$\bullet$ {\bf $A_{2k+1}$-singularity}
\medskip

The curve has two regular branches; we have $f=gh$ with $g,h$ germs of regular functions. 
As $\mu(f) = 2 m(g,h) - 1=2k+1$, it follows by Theorem \ref{teo:princ} that  
$I_f = I_g+I_h+6k+6$ and $V_f = V_g+V_h+12k+12$. Therefore $I_f\ge 6k+6$ and $V_f\ge 12k+12$.

Observe that if $I_hI_g\ne 0$, we have $V_g=I_g-1$ and $V_h=I_h-1$ as $g=0$ and $h=0$ are regular curves, so 
$$
V_f=(I_g-1)+(I_h-1)+12k+12=I_f+6k+4=I_f+3\mu(f)+1.
$$
When $I_g= 0$ (similarly for $I_h= 0$), the number of vertices of $g=0$ depends on $\lambda(g)$ and can take any integer value.

Taking $f(x,y) = (x+y^2+y^3-y^{k+1}) (x+y^2+y^3+y^{k+1})$, we see that 
the minimum values for $I_f$ and $V_f$ can be realized.
In fact, for any integers $m\ge 1$, there are curves $f=0$ with an $A_{2k+1}$-singularity and with
$I_f=6k+6+m$  (and $V_f$ as above); consider for example, the curve 
$f(x,y) = (x-y^{k+1}) (x+y^{k+m+1})=0$ which is of type $A_{2k}$ since assigning weights $wt\ x=k+1, wt\ y=1$, the terms of lowest weight are $x^2-xy^{k+1}$.

We obtain the following, where for simplicity, we left out the cases where a component has $n=2m,$ $a_j=0, 2m+1\le j\le 4m$ and $a_{4m}-a_{2m}^3=0$.
Of course there is an algorithm for computing $V_f$ where the parametrisations are known, provided by Theorem \ref{paramcase}.

\begin{theo}\label{theo:Simple_f}
Let $f:(\C^2,0)\to (\C,0)$ be a germ of a holomorphic curve with a simple singularity. Then the possible values of $I_f$ and $V_f$ for each $\mathcal K$-type of the singularity of $f$ are as in \mbox{\rm Table \ref{tab:IfVfSimpleSing}} where we make the assumption above for the values of $V_f$  of the $A_{2k}$ and $D_{2k+1}$-singularities. 
\end{theo}

\begin{table}[h!]
\begin{center}
	\caption{$I_f$ and $V_f$ for simple singularities of $f=0$ $(p,q\in \mathbb N)$.}
	\begin{tabular}{|c|c|c|}
		\hline
		Singularity of $f$ & $I_f$ & $V_f$  \\ 
		\hline
$A_{2k}$, $k \geq 1$ & $I_f=6k+2j-1$, $2 \leq j \leq k$,  & $V_f=12k+2j$, $2 \leq j \leq k$, \\
		& or $I_f=8k$ & or $V_f = 14k+1$\\
		\hline
$A_{2k+1}$, $k \geq 0$ & $I_f=p$, $6k+6 \leq p \leq \infty$ & $V_f=q$, $12k+12 \leq q\leq \infty$ \\ 
		\hline
$D_{2k+1}$, $k \geq 3$ &$I_f=p$,  $6k+9 \leq p \leq \infty$ &$V_f=q$,  $12k+16 \leq q \leq \infty$\\ \hline
$D_{2k}$, $k \geq 2$ &$I_f=p$,  $6k+6 \leq p \leq \infty$ &$V_f=q$,  $ 12k+12\leq q \leq \infty$\\ \hline
$E_6$ & $22$ & $43$\\ \hline
$E_7$ & $I_f=p$, $26 \leq p \leq \infty$ & $V_f=q$, $51 \leq q \leq \infty$ \\ \hline
$E_8$ & $29$ & $56$\\
		\hline
	\end{tabular}\label{tab:IfVfSimpleSing}
\end{center}

\end{table}

We turn now to parametrised curves with $\mathcal A$-simple singularities. These 
are classified in \cite{BruceGaffney} and are listed in Table \ref{tab:SingParam}. 
We compute below $I_f$ and $V_f$ for their defining equations; the $A_{2k}$ is already done.

\begin{table}[tp]
	\begin{center}
		\caption{$\mathcal A$-Simple singularities of parametrised curves (\cite{BruceGaffney}).}
		{\footnotesize
			\begin{tabular}{|c|l|}
				\hline
				Singularity of $f$ & $\mathcal A$-normal form \\ \hline
				$A_{2k}$ &$\left( t^2,t^{2k+1} \right)$ \\
				\hline
				$E_{6k}$ & $\left(t^3,t^{3k+1}+t^{3k+p+2}\right)$, $0\le p\le k-2$,\\
				&$\left(t^3,t^{3k+1}\right)$\\ 
				\hline
				$E_{6k+2}$& $\left(t^3,t^{3k+2}+t^{3k+p+4}\right)$, $0\le p\le k-2$,\\
				&$  \left(t^3,t^{3k+2}\right)$\\
				\hline
				$W_{12}$ &$(t^4,t^5+t^7)$; $(t^4,t^5)$\\
				\hline
				$W_{18}$ &$(t^4,t^7+t^9)$; $(t^4,t^7+t^{13})$; $(t^4,t^7)$\\
				\hline
				$W^{\#}_{1,2q-1}$ &$(t^4,t^6+t^{2q+5})$, $q\ge 1$\\
				\hline
			\end{tabular}\label{tab:SingParam}
		}	
	\end{center}
	
\end{table}

\medskip
$\bullet$ {\bf $E_{6k}$-singularity}
\medskip

For all the possible $\mathcal A$-orbits of $\gamma$ within the  $E_{6k}$-singularity, we can take a parametrisation of the form
$$
\gamma(t) = \left(t^3,a_6 t^6 +a_9 t^9+ \cdots +a_{3k} t^{3k}+a_{3k+1}t^{3k+1} + O(3k+2)\right),
$$
with $a_{3k+1}\ne 0$. As $m = 3$, we get
$$
I_\gamma = 
\left\{ 
\begin{array}{ll}
	3j& \textrm{if } a_{3p}=0, 2\le p\le j-1, a_{3j}\ne 0, 2\le j\le k.\\
	3k+1 & \textrm{otherwise}
\end{array}\right.
$$ 

Since $\mu(f) = 6k$, it follows by Theorem \ref{teo:princ} that $I_f = 18k+I_\gamma$, with $I_\gamma$ as above. Therefore, the values of  $I_f$ jump by $3$ except for the last jump which is by $1$.

For vertices we have 
$$
V_\gamma = \left\{ \begin{array}{ll}
3j+3 & a_{3p}=0, 3\le p\le j-1, a_{3j}\ne 0, 3\le j\le k,\\
3k+4 & \textrm{otherwise}
\end{array}\right.
$$
and $V_f = 36k+V_\gamma$, with $V_\gamma$ as above. The jumps in the values of $V_f$ are similar to those of $I_f$.

Here too, for simplicity we leave out the cases described above $a_{2m}\ne 0, a_{4m}-a_{2m}^3=0$ etc.)

\medskip
$\bullet$ {\bf $E_{6k+2}$-singularity}
\medskip

This is similar way to the $E_{6k}$-singularity case. 
We take a parametrisation  in the form 
$$
\gamma(t) = \left(t^3,a_6 t^6 +a_9 t^9+ \cdots +a_{3k} t^{3k}+a_{3k+2}t^{3k+2} + O(3k+3)\right),
$$
with $a_{3k+2}\ne 0$, so the only difference is that the last jump in values of $I_f$ and $V_f$ is by $2$.

\medskip
$\bullet$ {\bf $W_{12}$, $W_{18}$ and $W_{1,2q-1}^\#$-singularities}
\medskip

Here $m=4$ for all the singularities, $n_1 = 5$ for $W_{12}$, $n_1 = 7$ for $W_{18}$ and $n_1=6$ for $W^\#_{1,2q-1}$. 
Thus, $I_\gamma = n_1+1$ and $V_\gamma = n_1+6$.
As $\mu(W_{12})=12$, $\mu(W_{18})=18$, $\mu(W^\#_{1,2q-1})=2q+14$ (see \cite{Arnold}),
we obtain $I_f$ and $V_f$ by applying Theorem \ref{teo:princ}.

\begin{theo}
	Let $f$ be a germ of a holomorphic curve with a parametrisation $\gamma: (\C,0)\to (\C^2,0)$. If $\gamma$ has an $\mathcal A$-simple singularity, then the possible values for $I_f$ and $V_f$ are as in \mbox{\rm Table \ref{tab:IfVfSimpleSingParm}}. Here too for $V_f$ we assume that $ a_{4m}-a_{2m}^3 \neq 0$ when $n=2m$.
\end{theo}

\begin{table}[h]
	\begin{center}
\caption{$I_f$ and $V_f$ for equations of $\mathcal A$-simple singularities of parametrised curves.}
{\footnotesize
\begin{tabular}{|c|c|c|}
				\hline
Singularity of $f$ & $I_f$& $V_f$  \\ \hline
	$A_{2k}$, $k \geq 1$ & $I_f=6k+2j-1$, $2 \leq j \leq k$,  & $V_f=12k+2j$, $2 \leq j \leq k$, \\
	& or $I_f=8k$ & or $V_f = 14k+1$\\
	\hline
$E_{6k}$ & $I_f = 18k+3j$, $2 \leq j \leq k$ & $V_f=36k+3j+3$, $3 \leq j \leq k$\\ 
				& or $I_f = 21k+1$& or $V_f = 39k+4$ \\
\hline
$E_{6k+2}$&  $I_f = 18k+3j+6$, $2 \leq j \leq k$ & $V_f=36k+3j+15$, $3 \leq j \leq k$\\ 
                       & or $I_f = 21k+2$& or $V_f = 39k+5$ \\
\hline
$W_{12}$ & $42$ & $83$\\
\hline
$W_{18}$ & $62$ & $121$\\
\hline
$W^{\#}_{1,2q-1}$ & $2q+21$  & $2q+26$\\
\hline
\end{tabular}\label{tab:IfVfSimpleSingParm}
}	
	\end{center}
\end{table}

Finally we note that because of the form of $i_f$ we can compute $I_f$ for some germs directly from the formula $I_f=\dim\frac{ {\cal O}_2}{\langle f,i_f\rangle}$.
Suppose that $f:(\C^2,0)\to (\C,0)$ is semi-quasihomogeneous (SQH), with quasi-homogeneous part $g$. So for some coprime positive integer weights respectively $w_1, w_2$ for $x, y$, the polynomial $g$ has degree $d$, and all the terms of $f-g$ have weight $>d$.  We suppose that both weights are not $1$, that is $g$ is not homogeneous, and neither $x$ nor $y$ divides $g$. With the same weights, $i_f$ is also semi-weighted homogeneous of degree $3d-2w_1-2w_2$. Of course this does not calculate all possible values of $I_f$ in the $\mathcal K$-orbit of $f$.

\begin{prop}
With these assumptions $I_f=d(3d-2w_1-2w_2)/w_1w_2$, and this is the maximum value that can occur for any germ ${\cal K}$-equivalent to $f$.
\end{prop}

\begin{proof}
The map $(x,y)\mapsto (t^{w_1}x,t^{w_2}y)$ preserves $g=0$ and any components, so since $x, y$ do not divide $g$ it has no line components. It follows that $I_g$ is finite,  so $(f,i_f)$ is a semi-quasihomogeneous mapping with finite quasi-homogeneous part $(g,i_g)$, with the degree of the first term $d$ and the second $3d-2w_1-2w_2$. It follows that $m(f,i_f)=m(g,i_g)=d(d-2w_1-2w_2)/w_1w_2$ by the generalised Bezout formula, \cite{Arnold}, p. 200. For the second part since $g$ is quasi-homogeneous, so is any irreducible factor $g'$ and it is parametrised as $\gamma(t)=(\alpha t^{w_1},\beta t^{w_2})$ for some $\alpha, \beta, \alpha\beta\ne 0$. Assume that $w_1<w_2$, then from above $I_g=w_1+w_2-3$ and this is the maximum possible value for $I_{g'}$. The result then follows from Theorem 3.4 since if $g=g_1\ldots g_n$ the $m(g_i,g_j)=w_1w_2$ are the intersection numbers for the corresponding components of $f$ and each $I_{g_i}$ is maximal.
\end{proof}

\begin{ex}
{\rm
(1) Let $f$ be of type $A_{2k}$. By a similarity we can suppose that the lowest order terms for $f$ are $x^2$. Assigning weights ${\rm wt}(x)=k+1, {\rm wt}(y)=2$ suppose that $f$ is SQH with QH part $x^2+cy^{2k+1}$ with $c\ne 0$. Clearly $i_f=8k$, note that from Table \ref{tab:IfVfSimpleSing} this is the largest value possible for an $A_{2k}$-singularity.

(2) More challenging examples are provided by examples from the extensive lists of Arnold (see \cite{Arnold}), e.g. $W_{11}: x^4+y^5+ax^2y^3$; here $w_1=5, w_2=4, d=20$ and $I_f=42$, or for the $J_{k,0}$ singularities 
$$
f(x,y)=x^3+bx^2y^k+y^{3k}+(c_0+\ldots+c_{k-3}y^{k-3})xy^{2k+1}, 4b^3+27\ne 0, k\ge 3;
$$ 
here $w_1=k, w_2=1, d=3k$ and $I_f=3(7k-2)$.
}
\end{ex}

\section{Vertices and inflections of curves in $\R^2$} \label{sec:vertR}

Given a real analytic curve $f(x,y)=0$ in the Euclidean plane  $\R^2$, we can consider its complexification 
$f_{c}(x,y)=0$ in $\C^2$  and define $I_f$ (resp. $V_f$) as $I_{f_{c}}$ (resp. $V_{f_{c}}$)
Then $I_f$ (resp. $V_f$) gives an upper bound of the number of inflections (resp. vertices) concentrated at the singularity. 
Denote by ${\mathscr R}I_f$ (resp. ${\mathscr R}V_f$ )  the maximum number of real inflections (resp. vertices) concentrated at the singular point and that can appear when deforming the curve, 
so ${\mathscr R}I_f\le I_f$ and ${\mathscr R}V_f\le V_f$. If $f$ is merely smooth then provided $(f,i_f)$ (resp. $(f,v_f)$) is a finite mapping we still obtain a well-defined upper bound, and as we have seen these are finite maps for all $f$ off a set of infinite codimension.

It is shown in \cite{DiattaGiblin,MarcoSamuel} that for an $A_1^+$-singularity, ${\mathscr R}I_f=0$ and ${\mathscr R}V_f=4$ and  
for an $A_1^-$-singularity, ${\mathscr R}I_f=2$ and ${\mathscr R}V_f=6$. 
As we have $I_{f_c}=6$ and $V_{f_c}=12$ at an $A_1$-singularity, following the arguments 
in Example \ref{Ex:A1}, for each double point we need to remove at least 4 inflections and at least 6 vertices. 
This suggests that for an irreducible germ $f$,   
${\mathscr R}I_f\le I_f-4\delta= I_f-2\mu$ and  ${\mathscr R}V_f\le V_f-6\delta=V_f-3\mu$.

There is the concept of the degree of a real map germs $F:(\R^2,0)\to (\R^2,0)$ which 
is the degree of the mapping $F/||F||: S_{\epsilon}\to S_1$ where $S_1$ (resp $S_{\epsilon}$) is the oriented unit circle (resp. circle of radius $\epsilon$), with orientation as that of $\mathbb R^2$. If $F$ is differentiable, the degree is the sum of the signs
of the Jacobian of $F$ at all the pre-images of a regular value near
$0$; see \cite{Eisenbud} for a formula of the degree. 

If we consider the map-germ $F=(f,i_f)$, then an ordinary inflection has degree (or index) $\pm 1$. 
If we choose a frame given by the tangent and normal vectors to the curve, then the index is $+1$ (resp. $-1$) if the curve  lies in the first and third (resp. second and fourth) quadrants; see Figure \ref{fig:infIndVert}.

For vertices, considering the map-germ $F=(f,v_f)$, an ordinary vertex also has index $\pm 1$. 
In \cite{SalariTari} there is defined the notion of inward and outward vertices depending on the relative position of the evolute; see Figure \ref{fig:infIndVert}. We have an inward vertex if $\kappa\kappa''>0$ and an outward one otherwise. 
If we choose a frame given by the tangent and normal vectors to the curve, then the index is $+1$ at an inward vertex and $-1$ at an outward vertex when $\kappa(0)>0$ and vice versa 
when $\kappa(0)<0$.

For a singular germ, the degree of $(f,i_f)$ (resp. $(f,v_f)$) gives the sum of the indices of the ordinary inflections (resp. vertices) that appear in a deformation of $f$.
Of course the degree does not give information about the {\it number} of inflections or vertices that appear in the deformation.

\begin{figure}[tp]
\begin{center}
\includegraphics[scale=2]{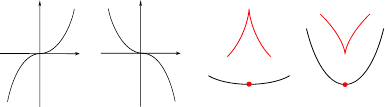}
\caption{The degree of an ordinary inflection, $+1$ (first figure) and $-1$ (second figure); 
	inward vertex (third figure) of degree $+1$  and outward vertex (last figure) of degree $-1$. The curve in red is the evolute of the curve in black.}
	\label{fig:infIndVert}
\end{center}
\end{figure}

\section{Appendix: Proof of Theorem \ref{theo:GenericIV}}\label{sec: Appendix}

\begin{proof}
	The proofs here are modelled on the account in \cite{wall}, p. 513, of Tougeron's proof that map-germs are ${\cal K}$-finite in general. We work in the complex case, but the real case  is essentially the same argument. First note that the $r$-jet of $i_f$ (resp. $v_f$) depend only on the $r+2$ (resp. $r+3$) jet of $f$. It follows that $T_e{\cal G}(f,i_f)+{\cal M}_2^{r}\{e_1,e_2\}$ (resp.  $T_e{\cal G}(f,v_f)+{\cal M}_2^{r}\{e_1,e_2\}$) depends only on the $r+2$, (resp. $r+3$) jet $z$ of $f$; denote its codimension by $d_e(z,I,\cal G)$ (resp. $d_e(z,V,\cal G)$). Define
	$$
	\begin{array}{c}
		W_I^r=\{z\in J^{r+2}(2,1):d(z,I,{\cal G})\ge r\}\\
		({\rm resp.}\,  W_V^r=\{z\in J^{r+3}(2,1):d(z,V,{\cal G})\ge r\})
	\end{array}
	$$
	
	These are clearly algebraic. If $(f,i_f)$ is ${\cal G}$-finite then $d_e(j^r(f,i_f),{\cal G})$ is eventually constant, so $j^s(f,i_f)\notin W_I^s$ for $s$ sufficiently large. Conversely if  $j^r(f,i_f)$ does not lie in $W_I^r$ then since $T_e{\cal G}(f)$ is an ${\cal O}_2$-module ${\cal M}_2^{r-1}.{\cal O}_2^2\subset T_e{\cal G}$ and $(f,i_f)$ is ${\cal G}$-finite. Which shows that the $f\in {\cal M}_2^2$ with $(f,i_f)$ not ${\cal G}$-finite form a pro-algebraic set. The same argument works for $(f,i_v)$, and indeed for $i_f, v_f$. We now need to show, in each case, that for any $k$-jet $z\in J^k(2,1)$ there is an $f$ with $j^kf=z$, and $(f,i_f)$ ${\cal G}$-finite; similarly for $(f,v_f)$, and for $i_f, v_f$ where we need to find $f$ with $j^kf=z$ and $i_f$ (resp. $v_f$) ${\cal K}$-finite.
	
	(1) Let $P_N$ be the set of polynomials in $x, y$ with terms of degree $d$ where $N+1\le d\le N+3$ and $F :(\C^2\times P_N,(0,0)) \to (\C^2,0)$
	defined by 
	$$
	F(x,y,q)=((f+q)(x,y),i_{f+q}(x,y)).
	$$
	
	We shall show that for almost all $q\in P_N$ in the sense of Lebesgue measure the germs $F_q:(\C^2,0)\to (\C^2,0), i_{f+q}:(\C^2,0)\to(\C,0)$, with $F_q(x,y)=F(x,y,q)$, are ${\cal K}$ finite. For the first
	choose $q'\in P_N$, and writing $g=f+q'$, consider
	
	\begin{eqnarray}\label{eq:dervMapFTransver}
		\lim_{s\to 0} \frac{F(x,y,q'+sq)-F(x,y,q')}{s}& = & (q,2g_{xx}g_yq_y+g_y^2q_{xx}-2g_{xy}(g_xq_y+g_yq_x) \nonumber\\
		&& -2g_xg_yq_{xy}+2g_{yy}g_xq_x+g_x^2q_{yy})\nonumber\\
		&=& (q,2(g_{yy}g_x-g_{xy}g_y)q_x+2(g_{xx}g_y-g_{xy}g_x)q_y\nonumber\\
		&& +g_y^2q_{xx}-2g_xg_yq_{xy}+g_x^2q_{yy}).\nonumber
	\end{eqnarray}

	It is not hard to show that for any $(x,y)\ne (0,0)$ the linear map $P_N\to \C^6$, given by 
	$$
	q\mapsto (q(x,y),q_x(x,y),q_y(x,y),q_{xx}(x,y),q_{xy}(x,y),q_{yy}(x,y))
	$$ 
	is surjective. So the derivative of $F$ at $((x,y),q')$, where $(x,y)\ne (0,0)$, is surjective unless $g_x=g_y=0$ which is not the case. So $F^{-1}(0,0)\setminus \{(0,0)\}$ is smooth; project to $P_N$ and choose a regular value $q$ of the projection. Then $j^k(f+q)=j^kf$ and $(f+q,i_{f+q})$ is $\mathcal K$-finite, by Theorem 2.1 in \cite{wall} (and hence ${\cal C}$-finite (\cite{wall}, p 513), so $I_{f+p}$ is finite). The same calculation shows for generic $q$ the germ $i_q:(\C^2,0)\to(\C,0)$ is ${\cal K}$-finite. Note that the proof shows that it is enough to consider $q'=0$.
	
	For vertices we consider instead the map $F(x,y,q)=((f+q)(x,y),v_{f+q}(x,y))$. Again, to simplify notation it suffices to work at $q=0$ and we consider
	$$
	\lim_{s\to 0} \frac{F(x,y,sq)-F(x,y,0)}{s}.
	$$ 
	
	Then its second component becomes
	{\footnotesize
		$$
		C_1q_x +C_2q_y+C_3 q_{xx}+C_4 q_{xy}+C_5 q_{yy}+
		(f_x^2 + f_y^2)( f_x^3q_{yyy} - 3f_x^2f_yq_{xyy}+ 3f_xf_y^2q_{xxy} - f_y^3q_{xxx} ),
		$$
	}
	with 
	{\footnotesize
		$$
		\begin{array}{rl}
			C_1=&5f_x^4f_{yyy} - 12f_x^3f_{xy}f_{yy} - 12f_x^3f_y f_{xyy}+ 9f_x^2f_yf_{xx}f_{yy} + 9f_x^2f_y^2 f_{xxy}+ 18f_x^2f_yf_{xy}^2 + 3f_x^2f_y^2f_{yyy} \\
			&- 9f_x^2f_yf_{yy}^2 
			- 18f_xf_y^2 f_{xx}f_{xy}- 2f_xf_y^3f_{xxx} + 18f_xf_y^2f_{xy}f_{yy} - 6f_xf_y^3f_{xyy} + 3f_y^3f_{xx}^2 \\
			&- 3f_y^3f_{xx}f_{yy} + 3f_y^4f_{xxy} - 6f_y^3f_{xy}^2\\
			C_2=&-3f_x^4f_{xyy} + 3f_x^3f_{xx}f_{yy} + 6f_x^3f_yf_{xxy} + 6f_x^3f_{xy}^2 + 2f_x^3f_yf_{yyy} - 3f_x^3f_{yy}^2 - 18f_x^2f_yf_{xx}f_{xy} \\
			&- 3f_x^2f_y^2f_{xxx} + 18f_x^2f_yf_{xy}f_{yy} - 9f_x^2f_y^2f_{xyy} + 9f_xf_y^2f_{xx}^2 - 9f_xf_y^2f_{xx}f_{yy}+ 12f_xf_y^3f_{xxy} \\
			& - 18f_xf_y^2f_{xy}^2 + 1f_y^32f_{xx}f_{xy} - 5f_y^4f_{xxx} \\
			C_3=&
			3f_y(f_x^3f_{yy} - 3f_x^2f_yf_{xy} + 2f_xf_y^2f_{xx} - f_xf_y^2f_{yy} + f_y^3f_{xy})\\
			C_4=&
			-3f_x^4f_{yy} + 12f_x^3f_yf_{xy} - 9f_x^2f_y^2f_{xx} + 9f_x^2f_y^2f_{yy} - 12f_xf_y^3f_{xy} + 3f_y^4f_{xx}\\
			C_5=&- 3f_x(f_x^3f_{xy} - f_x^2f_y f_{xx}+ 2f_x^2f_yf_{yy} - 3f_xf_y^2f_{xy} + f_y^3f_{xx})
		\end{array}
		$$
	}
	
The same argument as above for inflections shows that the derivative of $F$ at $((x,y),0)$, where $(x,y)\ne (0,0)$, is surjective if $(f_x^2+f_y^2)(x,y)\ne 0$. In the real case $f=f_x=f_y=0$ only at $(0,0)$. In the complex case if $f(x,y)=(f_x^2+f_y^2)(x,y)=0$ has solutions in every neighbourhood of $(0,0)$ it holds on a component of $f$ which we parametrise by $\gamma(t)=(x(t),y(t))$. Now differentiating $f\circ \gamma(t)\equiv 0$ we get $f_x(x(t),y(t))x'(t)+f_y(x(t),y(t))y'(t)\equiv 0$. We deduce that $x'(t)^2+y'(t)^2\equiv 0$. Integrating we see that we have a parametrisation of one of the complex lines $x=\pm iy$, the components of the degenerate complex circle $x^2+y^2=0$. So the argument is complete unless the germ $f$ has such a factor, but such germs are themselves of infinite codimension. The condition $f_x^2+f_y^2\equiv 0$ of course implies $f=g(x+iy)$ or $g(x-iy)$ for some function of $1$-variable $g$. Calculations become clearer if we change co-ordinates writing $x=(u+v)/2, y=(u-v)/2i, F(u,v)=f(x,y)$ so $f_x^2+f_y^2=F_uF_v$, and 
		$$
		i_f=F_{uu}F_v^2+F_{uv}(F_u^2+F_v^2)+F_{vv}F_u^2
		$$
		$$
		v_f=F_uF_vD(u,v)+\frac{3}{16}(F_{uu}F_v^2-F_{vv}F_u^2)(F_{uu}F_v^2-2F_{uv}F_uF_v+F_{vv}F_u^2),
		$$
		for some smooth $D$. One can check that $C_3=C_4=C_5=0$ precisely when $F_{uu}F_v^2+2F_{uv}F_uF_v+F_{vv}F_u^2=0$.

		To prove that $v_f$ is ${\cal K}$-finite in general we need to consider $v_f=f_x^2+f_y^2=0$. Over the reals $f_x=f_y=0$ already implies $x=y=0$. Over the complexes there are two cases; suppose first that $F_u=0$ ($F_v=0$ is the same argument).  Then $v_f=0$ if and only if $F_{uu}F_v=0$. We know that $F_u=F_v=0$ is just the origin so we need to ask when $F_u=0, F_{uu}=0$ have a component in common, when incidentally the $C_3, C_4, C_5$ terms above all vanish. This certainly means that $(F_u,F_{uu})$ is not finite; but Tougeron's argument as in \cite{wall} shows that the $F$ with $(F_u,F_{uu})$ not finite is also a pro-algebraic set of infinite codimension, so off this set we can apply the argument to prove that $v_f$ if ${\cal K}$-finite in general.
	
	\medskip
Part (2) follows immediately from (1), while (3) is proved above.
	
	\medskip
It remains to prove (4). Since $f$ is ${\cal K}$-finite for some $k$ the orbit of $f$ contains $f+{\cal M}_2^k$, so is not contained in any proalgebraic set of infinite codimension. In fact we see that in a very strong sense almost any $g$ which is ${\cal K}$-equivalent to $f$ has these properties.
\end{proof}

\begin{acknow}
The second author was supported by the FAPESP postdoctoral grant 
 2022/06325-8.
The work in this paper was partially supported by the FAPESP Thematic project grant 2019/07316-0.
\end{acknow}


\noindent 
JWB: Department of Mathematical Sciences, University of Liverpool, Liverpool, L69 3BXl\\
E-mail: billbrucesingular@gmail.com\\

\noindent
MACF: Departamento de Matem\'atica, Universidade Federal de Vi\c{c}osa, Av. P. H. Rolffs, s.n., Campus Universit\'ario, CEP: 36570-000, Vi\c{c}osa - MG, Brazil
\\
E-mail: marco.a.fernandes@ufv.br\\

\noindent 
FT: Instituto de Ci\^encias Matem\'aticas e de Computa\c{c}\~ao - USP, Avenida Trabalhador s\~ao-carlense, 400, Centro, CEP: 13566-590, S\~ao Carlos - SP, Brazil.
\\
E-mail: faridtari@icmc.usp.br

\end{document}